\documentclass[twoside]{amsart}
\usepackage{mathrsfs}
\usepackage{amsfonts}
\usepackage{amssymb}
\usepackage{amssymb}
\usepackage{amssymb}
\usepackage{amssymb}
\usepackage{amssymb}
\usepackage{amssymb}
\usepackage{amssymb}
\usepackage{amssymb}
\usepackage{amsmath}
\usepackage{mathdots}
\usepackage{xcolor}
\usepackage[all,cmtip]{xy}

\usepackage{lastpage}

\RequirePackage{amsmath} \RequirePackage{amssymb}
\usepackage{amscd,latexsym,amsthm,amsfonts,amssymb,amsmath,amsxtra}
\usepackage[colorlinks=true, urlcolor=blue, linkcolor=red,anchorcolor=blue,citecolor=blue]{hyper ref}

    \newcommand{\BC}{{\mathbb {C}}}

        \newcommand{\bC}{{\textbf {C}}}

     \newcommand{\CF}{{\mathcal {F}}}

    \newcommand{\CS}{{\mathcal {S}}}

    \newcommand{\lenth}{{\mathrm {\lenth}}}

    \newcommand{\Hom}{{\mathrm{Hom}}} 
    \newcommand{\Ind}{{\mathrm{Ind}}}

    \theoremstyle{plain}

    \newtheorem{thm}{Theorem}[section] 
      \newtheorem{proposition}[thm]{Proposition}

    \numberwithin{equation}{section}

\usepackage[top=1in, bottom=1in, left=1.25in, right=1.25in]{geometry}

\begin{document}
\title{A note on small theta lift}

\begin{abstract}
In this note, we use certain sesquilienar form to realize small theta lift for even orthogonal-symplectic and unitary dual pairs over p-adic fields.
\end{abstract}

	\author{Jingsong Chai}
	\address{School of Mathematics, Physics and Finance \\ Anhui Polytechnic University \\Wuhu, Anhui,  241000\\China}
	\email{jingsongchai@hotmail.com}


	\maketitle
\section{Introduction}

Let $(G',G)$ be an irreducible type I reductive dual pair inside the symplectic group $Sp(F)$, where $F$ is a p-adic field. Let $\tilde{S}p(F)$ be the metaplectic two fold cover of $Sp(F)$. For any subgroup $H$ of $Sp(F)$ we let $\tilde{H}$ denote its inverse image in $\tilde{S}p(F)$. Let $(\omega, \mathcal{Y})$ be the Weil representation of $\tilde{Sp}(F)$. Let $(\pi', V_{\pi'})$ be any irreducible unitary representation of $\tilde{G}'$. Let $\mathcal{Y}^\infty, H_{\pi'}^\infty$ denote the space of smooth vectors. Then $\tilde{G}'$ acts on the algebraic tensor product $\mathcal{Y}^\infty\otimes V_{\pi'}^\infty$ by $\omega\otimes \pi'$. The space $\mathcal{Y}^\infty\otimes V_{\pi'}^\infty$ is naturally endowed with an inner product $<,>$ coming from the unitary structure of $\mathcal{Y}$ and $H_{\pi'}$. Consider the new sesquilinear form $<,>_{\pi'}$ defined by
\[
<\Phi,\Phi'>_{\pi'}=\int_{\tilde{G}'}<\Phi, (\omega\otimes \pi')(g')\Phi'>dg'
\]
for $\Phi,\Phi'\in \mathcal{Y}^\infty\otimes V_{\pi'}^\infty$. Assume $(G',G)$ and $\pi'$ are such that the above integral converges for all $\Phi,\Phi'$. Let $R$ denote the radical of $<,>_{\pi'}$. Then $<,>_{\pi'}$ factors to define a nondegenerate sesquilinear form on
\[
H(\pi'):=(\mathcal{Y}^\infty \otimes V_{\pi'}^\infty)/R.
\]
It is clear that $\tilde{G}$ acting on the first factor of $\mathcal{Y}^\infty\otimes V_{\pi'}^\infty$ preserves $<,>_{\pi'}$. Hence we have an action of $\tilde{G}$ on $V_{\pi'}$. The following conjecture is proposed by Li in \cite{L90}.  \\

$\mathbf{Conjecture}$. Suppose $<,>_{\pi'}$ is defined and nonzero. Then

(i) $<,>_{\pi'}$ is nonnegative, hence $H(\pi')$ is unitary.

(ii) $H(\pi')$ defines an irreducible unitary representation on the completion of $H(\pi')$ with respect to the inner product given by $(i)$.

(iii) The map $\pi'\to H(\pi'^\vee)$, where $\pi'^\vee$ denote the contragredient of $\pi'$, coincides with Howe's duality correspondence.  \\

Now let $(G',G)$ be either even orthogonal-symplectic or unitary dual pair. Let $\pi$ be an irreducible smooth representation of $G'$ with its contragredient $\pi^\vee$. Let's ignore the splitting of $G'\times G$ inside $\tilde{Sp}(F)$ for simplicity. Choose a nonzero element $\ell\in \Hom_{G'(W)\times G'(W)\times G(V)} (\omega\otimes \bar{\omega} \otimes \pi \otimes \pi^\vee , \BC)$. Define
\[
R:=\{ \varphi\otimes v^\vee\in \omega\otimes \pi^\vee: \ell(\varphi\otimes \varphi'\otimes v\otimes v^\vee)=0 \ \ \forall \varphi'\in \bar{\omega}, v\in \pi  \}
\]
be the radical of $\ell$. Put $H_{\ell,\psi}(\pi):= (\omega\otimes \pi^\vee)/R$.
The main result of the note is the following.
\begin{thm}
 $ H_{\ell,\psi}(\pi)$ is isomorphic to the small theta lift $\theta_\psi(\pi)$.
\end{thm}

$\bullet$ We have to restrict to even orthogonal-symplectic and unitary daul pairs since in the proof we used a result of Droschl(\cite{D23}), which is proved only for these dual pairs. Our method works in principal for other dual pairs as well once Droschl's result is available. \\

$\bullet$ In Li's conjecture, the sesquilinear form $<,>_{\pi'}$ is an example of $\ell$ in the above theorem. Thus our result can be viewed as an extension of part (ii) and (iii) in Li's conjecture to general irreducible smooth representations.  \\

\noindent \textbf{Acknowledgement:} We would like to thank Lu Hengfei for helpful discussions. This work is supported by a start up funding of AHPU and NSFC(no.12571010).
\\

\section{Proof of the main result}

Let $E$ be either $F$ or a quadratic extension of $F$. We fix a nontrivial additive character $\psi$ of F. Let $V$ and $W$ be vector spaces over $E$ equipped with nondegenerate sesquilinear forms
\[
(\cdot,\cdot):V\times V \to E \ \ \ \ \ \ \and \ \ \ \ \ \ <\cdot,\cdot>:W\times W\to E
\]
of opposite signs. Put
\[
m=\dim V  \ \ \ \ \ \ \and \ \ \ \ \ \ n=\dim W.
\]
We will consider the dual pair $(G'(W), G(V))$ arising from $V, W$, and will distinguish certain cases in the following table.

\begin{center}
\begin{tabular}{|c c c c c|} 
 \hline
 Cases &  & $G'(W)$ & $G(V)$ & $l$ \\ [0.5ex] 
 \hline
 $A$ & [E:F]=2 & $U_m$ & $U_n$ & $n-m$ \\ 
 \hline
 $C''$ & m even & $O_m$ & $Sp_n$ & $n-m+1$ \\
 \hline
 $D$ & n even & $Sp_m$ & $O_n$ & $n-m-1$ \\ [1ex] 
 \hline
\end{tabular}
\end{center}

Let $\boldsymbol{\chi}=(\chi_V,\chi_W)$ be a pair of characters of $E^\times$ as in section 3.2 in \cite{GI14}. Associated to these datum, there is a splitting
\[
\iota=\iota_{V,W,\boldsymbol{\chi},\psi}: G(W) \times H(V) \to Mp(\mathbb{W}),
\]
where $\mathbb{W}=V\otimes_E W$.

Let $\omega_\psi$ be the Weil representation of $Mp(\mathbb{W})$. The above splitting then induces a representation, denoted as $\omega_{V,W,\boldsymbol{\chi},\psi}$, of $G(W)\times H(V)$. If $\pi$ is an irreducible smooth (genuine) representation of $G'(W)$, the maximal $\pi$-isotypic quotient of $\omega_{V,W,\boldsymbol{\chi},\psi}$ is of the form $\pi\otimes \Theta_{V,W,\boldsymbol{\chi},\psi}(\pi)$, where $\Theta_{V,W,\boldsymbol{\chi},\psi}(\pi)$ is a smooth representation of $H(V)$. We will call $\Theta_{V,W,\boldsymbol{\chi},\psi}(\pi)$ as the big theta lift of $\pi$. It is always of finite length. We use $\theta_{V,W,\boldsymbol{\chi},\psi}(\pi)$ to denote the maximal semisimple quotient of $\Theta_{V,W,\boldsymbol{\chi},\psi}(\pi)$, and call it small theta lift of $\pi$.

The following fundamental result, called Howe duality, was conjectured by Howe (\cite{How79,How89}). It was proved by Waldspuger (\cite{Wal90}) when $p\neq 2$, and in full generality by Gan-Takeda (\cite{GT16}) and Gan-Sun (\cite{GS17}).

\begin{thm}
$\theta_{V,W,\boldsymbol{\chi},\psi}(\pi)$ is either zero or irreducible. Moreover, if for irreducible smooth representations $\pi_1, \pi_2$ of $G'(W)$, we have $\theta_{V,W,\boldsymbol{\chi},\psi}(\pi_1)\cong \theta_{V,W,\boldsymbol{\chi},\psi}(\pi_2)$, then $\pi_1\cong \pi_2$.
\end{thm}

Let $W_{-}$ denote the vector space $W$ over $E$ equipped with form $-<\cdot,\cdot>$. Let $\textbf{W}=W\oplus W_{-}$. Use $\textbf{G'(W)}$ to denote the isometry group of $\textbf{W}$, and we have an embedding
\[
i: G'(W)\times G'(W) \to \textbf{G(W)}.
\]  
Note that we have 
\[
\omega_{V,\textbf{W},\boldsymbol{\chi},\psi}\circ i\cong \omega_{V,W,\boldsymbol{\chi},\psi}\otimes (\bar{\omega}_{V,W,\boldsymbol{\chi},\psi}\cdot \chi_V)
\]
as representations of $G'(W)\times G'(W)$.

Consider the following see-saw diagram

\begin{displaymath}
\xymatrix{
\textbf{G'(W)} \ar@{-}[d] \ar@{-}[rd] & G(V)\times G(V) \ar@{-}[d]\\
G'(W)\times G'(W) \ar@{-}[ru]
&\Delta G(V)}
\end{displaymath}
Then the see-saw identity says that
\begin{eqnarray}
\label{see-saw}
&\Hom_{G(V)} (\Theta_{V,W,\boldsymbol{\chi},\psi}(\pi)\otimes \Theta_{V,W,\boldsymbol{\chi},\bar{\psi}}(\pi^\vee \chi_V), \chi_W ) \nonumber  \\
&\cong \Hom_{G'(W)\times G'(W)}(\Theta_{V,\textbf{W},\boldsymbol{\chi},\psi}(\chi_W), \pi\otimes \pi^\vee \chi_V ).
\end{eqnarray}

Inside $\textbf{W}$, set
\[
W^\triangle =\{ (w,w)\in \textbf{W}:w\in W \}, \ \ \ \ W^\nabla=\{(w,-w)\in \textbf{W}: w\in W \}.
\]
Then $W^\triangle$ is a maximal isotropic subspace, and we have a complete polarization
\[
\textbf{W}=W^\triangle \oplus W^\nabla.
\]
The stabilizer of $W^\triangle$ in $\textbf{G'(W)}$ is a Siegel parabolic subgroup $\textbf{P}=M_{\textbf{P}}U_\textbf{P}$. For $s\in \bC$ and a character $\chi$ of $E^\times$, let
\[
I_{\textbf{P}}^{\textbf{G'}}(s,\chi)=\Ind_{\textbf{P}}^{\textbf{G'(W)}} (\chi |\cdot|_E^s)
\]
denote the normalized induced representation of $\textbf{G'}=\textbf{G'(W)}$, where $\chi |\cdot|^s$ is a character of $M_{\textbf{P}}$ via 
\xymatrix{
M_{\textbf{P}}=GL(W^\triangle) \ar[r]^-{\det} &E^\times \\} 
and is extended to $\textbf{P}$ with $U_{\textbf{P}}$ acting trivially.

The Weil representation $\omega_{V,\textbf{W},\boldsymbol{\chi},\psi}$ of $\textbf{G'(W)}\times G(V)$ is realized on $\CS(V\otimes W^\nabla)$, here $'\CS'$ denote the space of Schwartz function space. For $h\in H, \phi\in \CS(V\otimes W^\nabla)$, the action of $G(V)$ is given by
\[
(\omega_{V,\textbf{W},\boldsymbol{\chi},\psi}(h))\phi(x)=\chi_W(\det h)\phi(h^{-1}x).
\]
For $\textbf{g}\in \textbf{G'(W)}$, put
\[
\CF_\phi(\textbf{g})=(\omega_{V,\textbf{W},\boldsymbol{\chi},\psi}(\textbf{g})\phi)(0).
\]
One can check that $\phi\to \CF_\phi$ is a $\textbf{G'}$-equivariant map
\[
\omega_{V,\textbf{W},\boldsymbol{\chi},\psi} \to I_{\textbf{P}}^{\textbf{G'}}(-\frac{l}{2},\chi_V).
\]
Let $R_{\textbf{W},\boldsymbol{\chi},\psi}(V,\chi_W)$ be the image of this map. The following result is due to Rallis (\cite{Ra84}).

\begin{proposition}
The map $\phi\to \CF_\phi$ induces an isomorphism
\begin{displaymath}
\Theta_{V,\textbf{W},\boldsymbol{\chi},\psi}(\chi_W)\stackrel{\cong} {\longrightarrow} R_{\textbf{W},\boldsymbol{\chi},\psi}(V,\chi_W) \subset I_{\textbf{P}}^{\textbf{G'}}(-\frac{l}{2},\chi_V).
\end{displaymath}
\end{proposition}

We have the following multiplicity one result by Droschl (\cite{D23}).

\begin{proposition}
For any irreducible smooth representation $\pi$ of $G'(W)$, and any character $\chi$ of $E^\times$, we have
\[
\dim_\BC \Hom_{G'(W)\times G'(W)} (I_{\textbf{P}}^{\textbf{G}}(s,\chi), \pi\otimes \pi^\vee \chi)=1.
\]
\end{proposition}

Note that
\begin{align*}
& \Hom_{G'(W)\times G'(W)\times G(V)} (\omega_{V,W,\boldsymbol{\chi},\psi}\otimes \bar{\omega}_{V,W,\boldsymbol{\chi},\psi} \otimes \pi \chi_V^{-1}\otimes \pi^\vee , \chi_W ) \\ 
&= \Hom_{G(V)} (\Theta_{V,W,\boldsymbol{\chi},\psi}(\pi)\otimes \Theta_{V,W,\boldsymbol{\chi},\bar{\psi}}(\pi^\vee \chi_V), \chi_W )  \\
&= \Hom_{G'(W)\times G'(W)}(\Theta_{V,\textbf{W},\boldsymbol{\chi},\psi}(\chi_W), \pi\otimes \pi^\vee \chi_V ).
\end{align*}
By Rallis' result and the above multiplicity one theorem, the dimension of these Hom spaces is at most one, and it is equal to one precisely when $\Theta_{V,W,\boldsymbol{\chi},\psi}(\pi)\neq 0$ by Lemma 6.1 in \cite{GI14}. 

Take a nonzero element $\ell\in \Hom_{G'(W)\times G'(W)\times G(V)} (\omega_{V,W,\boldsymbol{\chi},\psi}\otimes \bar{\omega}_{V,W,\boldsymbol{\chi},\psi} \otimes \pi \chi_V^{-1}\otimes \pi^\vee , \chi_W ) $. Let 
\[
R:=\{ \varphi\otimes v^\vee\in \omega_{V,W,\boldsymbol{\chi},\psi}\otimes \pi^\vee: \ell(\varphi\otimes \varphi'\otimes v\otimes v^\vee)=0 \ \ \forall \varphi'\in \bar{\omega}_{V,W,\boldsymbol{\chi},\psi}, v\in \pi\chi_V^{-1}  \}
\]
be the radical of $\ell$. Put $H_{\ell,\psi}(\pi):= (\omega_{V,W,\boldsymbol{\chi},\psi}\otimes \pi^\vee)/R$.

\begin{thm}
We have 
\[
H_{\ell,\psi}(\pi)\cong \theta_{V,W,\boldsymbol{\chi},\psi}(\pi).
\]
\end{thm}
\begin{proof}
Since 
\[
\Theta_{V,W,\boldsymbol{\chi},\psi}(\pi)=(\omega_{V,W,\boldsymbol{\chi},\psi}\otimes \pi^\vee)_{G(V)},
\]
it follows that $H_{\ell,\psi}(\pi)$ is a quotient of $\Theta_{V,W,\boldsymbol{\chi},\psi}(\pi)$. Hence it is of finite length. By the Howe duality, $\Theta_{V,W,\boldsymbol{\chi},\psi}(\pi)$ has a unique quotient $\theta_{V,W,\boldsymbol{\chi},\psi}(\pi)$. Thus $\theta_{V,W,\boldsymbol{\chi},\psi}(\pi)$ is also the unique quotient of $H_{\ell,\psi}(\pi)$. Hence we can write
\[
\theta_{V,W,\boldsymbol{\chi},\psi}(\pi) = (\omega_{V,W,\boldsymbol{\chi},\psi}\otimes \pi^\vee)/R_1,
\]
where $R\subseteq R_1\subset \omega_{V,W,\boldsymbol{\chi},\psi}\otimes \pi^\vee$.

If $H_{\ell,\psi}(\pi)$ is not irreducible, then $R_1$ will contain $R$ properly. Now consider
\[
H_{\ell, \bar{\psi}}(\pi^\vee \chi_V):=(\bar{\omega}_{V,W,\boldsymbol{\chi},\psi} \otimes \pi \chi_V^{-1} )/R',
\]
where $R'\subset \bar{\omega}_{V,W,\boldsymbol{\chi},\psi} \otimes \pi \chi_V^{-1}$ is defined similarly as $R$. Also $\theta_{V,W,\boldsymbol{\chi},\bar{\psi}}(\pi^\vee \chi_V)$ is a quotient of $H_{\ell,\bar{\psi}}(\pi^\vee \chi_V)$ and can be written as 
\[
\theta_{V,W,\boldsymbol{\chi},\bar{\psi}}(\pi^\vee \chi_V)=(\bar{\omega}_{V,W,\boldsymbol{\chi},\psi} \otimes \pi \chi_V^{-1} )/R_1'
\]
with $R'\subseteq R_1' \subset \bar{\omega}_{V,W,\boldsymbol{\chi},\psi} \otimes \pi \chi_V^{-1} $.

By Corollary 18.4 in \cite{GKT}, there is a natural nonzero linear form in 
\[
\Hom_{G(V)}(\theta_{V,W,\boldsymbol{\chi},\psi}(\pi)\otimes \theta_{V,W,\boldsymbol{\chi},\bar{\psi}}(\pi^\vee \chi_V), \chi_W).
\]
View this nonzero linear form as an element in 
\[
\Hom_{G(V)}((\omega_{V,W,\boldsymbol{\chi},\psi}\otimes \pi^\vee)/R_1 \otimes (\bar{\omega}_{V,W,\boldsymbol{\chi},\psi} \otimes \pi \chi_V^{-1} )/R_1' ,\chi_W),
\]
which can be naturally extended to a nonzero liner form in 
\[
\Hom_{G(V)}( H_{\ell,\psi}(\pi)\otimes H_{\ell,\bar{\psi}}(\pi^\vee \chi_V),\chi_W ).
\]
Extend this nonzero linear form further to an element in
\[
\Hom_{G(V)}( \Theta_{V,W,\boldsymbol{\chi},\psi}(\pi)\otimes \Theta_{V,W,\boldsymbol{\chi},\bar{\psi}}(\pi^\vee \chi_V), \chi_W  ).
\]
This nonzero element will correspond to a nonzero element $\ell'$ in
\[
\Hom_{G'(W)\times G'(W)\times G(V)} (\omega_{V,W,\boldsymbol{\chi},\psi}\otimes \bar{\omega}_{V,W,\boldsymbol{\chi},\psi} \otimes \pi \chi_V^{-1}\otimes \pi^\vee , \chi_W ).
\]
As this Hom space has dimention at most one, there exists a nonzero scalar $\lambda$, such that
\[
\ell=\lambda \ell'.
\]
But, by the construction of $\ell'$, the radical is $R_1$, which is strictly larger than $R$, and we got a contradiction, which will imply that $H_{\ell,\psi}(\pi)$ is irreducible.

So $H_{\ell,\psi}(\pi)$ is an irreducible quotient of $\Theta_{V,W,\boldsymbol{\chi},\psi}(\pi)$, it has to be $\theta_{V,W,\boldsymbol{\chi},\psi}(\pi)$ by Howe duality, and the conclusion follows.
\end{proof}



\begin{thebibliography}{SK}


\bibitem[D23]{D23}
J.Droschl. {\em Proof of a conjecture of Kudla and Rallis on quotients of degenerate principal series.}  \textit{Advances in Mathematics.} Vol.464, 110145, 2025.

\bibitem[GI14]{GI14}
Wee Tech Gan and Atsushi Ichino. {\em Formal degrees and local theta correspondence.} \textit{Inventiones Mathematicae.} Vol 195, 509-672, 2014.

\bibitem[GKT]{GKT}
Wee Tech Gan, S. Kudla and S. Takeda. {\em The local theta correspondence.} Preprint.

\bibitem[GS17]{GS17}
Wee Teck Gan and Binyong Sun. {\em The Howe duality conjecture: quaternionic case.} \textit{Representation theory, number theory, and invariant theory.} Progr. Math., 323, Birkhauser/Springer, Cham, 175-192, 2017.

\bibitem[GT16]{GT16}
Wee Teck Gan and S. Takeda. {\em A proof of the Howe duality conjecture.} \textit{J. Amer. Math. Soc.} 29, no.2, 473-493. 2016

\bibitem[How79]{How79}
Roger Howe. {\em $\theta$-series and invariant theory.} \textit{Automorphic forms, representations and L-functions.} Proc. Sympos. Pure Math.,XXXIII, Amer. Math. Soc., Providence, R.I., 1979, 275-285.

\bibitem[How89]{How89} 
Roger Howe. {\em Transcending classical invaraint theory.} \textit{J. Amer. Math. Soc.} 2, no.3, 535-552, 1989.

\bibitem[L90]{L90}
Jianshu Li. {\em Theta lifting for unitary representations with nonzero cohomology.} \textit{Duke. Math. Journal} Vol.61, no.3, 913-937, 1990


\bibitem[Wal90]{Wal90}
J.-L. Waldspurger. {\em Démonstration d’une conjecture de dualité de Howe dans le cas p-adique, $p\neq 2$.} Festschrift in honor of I.I.Piatetski-Shapiro on the occasion of his sixtieth birthday, Part I (Ramat Aviv, 1989), Israel Math. Conf. Proc., 2, Weizmann, Jerusalem, 267-324, 1990.

\bibitem[Ra84]{Ra84}
S.Rallis. {\em On the Howe duality conjecture.} \textit{Compos. Math.} 51, 339-399, 1984.

\end{thebibliography}
\end{document}